\documentclass[vecarrow]{svmult}

\usepackage{makeidx}         % allows index generation
\usepackage{graphicx}        % standard LaTeX graphics tool
\usepackage{psfrag}                             % when including figure files
\usepackage{multicol}        % used for the two-column index
\usepackage{amsmath}
\usepackage{amssymb}
\usepackage{amsfonts}
\usepackage{latexsym}
\usepackage[T1]{fontenc}
\usepackage{subfigure}                      % figures and subfigures
\usepackage{float}
\usepackage{mathrsfs}
\usepackage{mathptmx}         % changes fonts to 
\usepackage{color}
\usepackage{hyperref}
\begin{document}

\title{Some convergence results on the periodic Unfolding Operator in
 Orlicz setting}
\index{Unfolding Method}

\titlerunning{Unfolding in Orlicz setting}

\author{Joel Fotso Tachago, Giuliano Gargiulo, Hubert Nnang, Elvira Zappale}

\authorrunning{J. Fosto Tachago, G. Gargiulo, H. Nnang, E. Zappale}

\institute{J. Fotso Tachago\at University of Bamenda, Bamenda, Cameroon,\hfill\break
\email{fotsotachago@yahoo.fr}
\and
G. Gargiulo\at  University of Sannio, Benevento, Italy, \hfill\break
\email{ggargiul@unisannio.it}, 
\and
H. Nnang\at University of Yaounde I, Yaounde, Cameroon,
\hfill\break
\email{hnnang@yahoo.fr}\and
E. Zappale \at Sapienza- University of Rome, Rome, Italy,\hfill\break
\email{elvira.zappale@uniroma1.it}}

\maketitle

\index{integral!equations, nonuniqueness}
\index{plane!elasticity}

\bgroup

% The following are the author's own macros. Use your own ONLY if you are
% conversant with LaTeX and know what you are doing. There must be no clash
% between your macros and the publisher's. To avoid this, use the construction
% with \bgroup (see immediately above) and \par\egroup (see at the very end
% of this specimen paper). This will make your macros local.

\renewcommand{\b}{\beta}
\renewcommand{\c}{\gamma}
\newcommand{\n}{\nu}
\newcommand{\x}{\xi}
\renewcommand{\l}{\lambda}
\newcommand{\m}{\mu}
\newcommand{\h}{\chi}
\newcommand{\s}{\sigma}
\renewcommand{\t}{\theta}
\newcommand{\trans}{^{\rm T}}
\newcommand{\ds}{{\partial S}}
\renewcommand{\o}{\omega}
\newcommand{\bb}{\mathcal B}
\newcommand{\rr}{\mathcal R}
\renewcommand{\ll}{\mathcal L}
\renewcommand{\ss}{\mathcal S}
\renewcommand{\gg}{\mathcal G}
\newcommand{\vv}{\mathcal V}
\newcommand{\ww}{\mathcal W}
\newcommand{\Dminus}{(D$^{\omega-}$)}
\newcommand{\Nminus}{(N$^{\omega-}$)}
\newcommand{\Rminus}{(R$^{\omega-}$)}
\newcommand{\Vminus}{{\mathcal V}^{\omega-}}
\renewcommand{\Re}{\mathop\textrm{Re}}
\renewcommand{\Im}{\mathop{\rm Im}}
\newcommand{\mxy}{|x-y|}
\def\Wminus{{\mathcal W}^{\omega-}}
\def\regular#1{C^2(S^{#1})\cap C^1(\bar{S}^{#1})}
\def\holder#1{C^{\,\,#1,\alpha}(\partial S)}
\def\ppd#1{\dfrac{\partial}{\partial#1}}
\newcommand{\e}{\varepsilon}
\newcommand{\supsig}{^{(\sigma)}}
\newcommand{\up}{\upsilon}
\newcommand{\Up}{\Upsilon}
\newcommand{\p}{\psi}
\renewcommand{\a}{\alpha}
\newcommand{\f}{\varphi}
\newcommand{\F}{\Phi}
\renewcommand{\P}{\Psi}
\renewcommand{\d}{\delta}
\renewcommand{\k}{\kappa}
\newcommand{\sums}{\sum_{\s=1}^2}
\newcommand{\summ}{\sum_{m=0}^\infty}
\newcommand{\intl}{\int\limits}
\newcommand{\pd}{\partial}
\renewcommand{\D}{\Delta}

\newcommand{\Cc}{{\mathcal C}}
\newcommand{\Dd}{{\mathcal D}}
\newcommand{\Fill}{\makebox[6pt]{\hfill}}

\newcommand{\tr}{{\mathrm T}}

\allowdisplaybreaks

\abstract*{The matrix characterizing nonuniqueness of solution for Fredholm
integral equations of the first kind is constructed and approximated numerically
in the case of plane elastic deformations.}

\section{Introduction}\label{sec:con1}
Homogenization of periodic structures via two-scale convergence in
Orlicz setting was introduced in \cite{fotso nnang 2012} and later expanded in \cite{FNZIMSE, FNZOpuscula, FTGNZ} in order to deal with convex integral functionals, multiscale problems, and differential operators.  

With the aim of making two-scale convergence a convergence in a function space, the unfolding method has been introduced in \cite{CDG1, CDG2, CDG3, CDG4}.

The scope of this note is to introduce the unfolding operator in \cite{CDG1, CDG2} to the Orlicz setting,
% i.e. defining a convergence in a functional space,  to adapt result of unfolding method to
%Orlicz-setting. 
having in mind that these spaces generalize classical $L^{p} $ spaces and capture more information than those available in the latter setting.

%present an application of new approach on quasiconvex integrals, wich
%describe various physical situations as the behavior of cellular elastic
%materials with very fine structure.

Let $m,d$ be positive integers, $Y=\left] 0,1\right[^d$ and let $\mathcal{%
	A}_{0}$ be the class of all bounded open subsets of $
\mathbb{R}^d$ %with Lipschitz boundary 
and let 
$B$ be an $N-$function satisfying $\nabla_2$ and $\triangle_{2} $ conditions (cf. Subsection \ref{O-spaces} for definitions).

We intend for each $\Omega \in \mathcal A_0(\mathbb R^s)$, i.e. $\Omega$ bounded open subset of $\mathbb R^d$, %with Lipschitz boundary, 
and for every sequence $\left\{
\varepsilon\right\} \subseteq \left] 0,+\infty \right[ $ converging to $%
0,$ to define the unfolding operator in $L^B(\Omega;\mathbb R^m)$ and make a parallel with the analogous notion and convergence in $L^p(\Omega;\mathbb R^m)$. Since we can argue  componentwise, we will assume in the sequel, without loss of generality that $m=1$.

To this end, in subsection 1.2  we deal with preliminaries on Orlicz
spaces, while the remaining of the paper is devoted to define the unfolding  method in this framework and prove some convergence results, establishing parallels with the strong and weak convergence in $L^B(\Omega \times Y)$ and $L^B(\Omega)$ between unfolded sequences and their original counterparts, thus extending to the Orlicz setting the result proved by \cite{CDG1, CDG2, CDG3, CDG4}, leaving the parallel with the notion of two scale convergence in the standard and Orlicz setting (cf. \cite{ngu0, All1, LNW, V, fotso nnang 2012}), as well as the applications, the extension to more scales and to higher order derivatives as in \cite{All2, FZ, nnang reit, FNZIMSE, FNZOpuscula, FTGNZ} for a forthcoming paper. 

 Our main theorem is indeed the following: 
 
 \begin{theorem}\label{convBunf}
 	Let $B$ be an $N$-function satisfying $\nabla_2$ and $\Delta_2$ conditions. The following results hold:
 	\begin{itemize}
 		\item[i)] For $w\in L^{B}( \Omega) ,\mathcal{T}%
 		_{\varepsilon }\left( w\right) \to w$ strongly in $L^{B}\left(
 		\Omega \times Y\right) $.
 		
 		\item[ii)] Let $\left\{ w_{\varepsilon }\right\}_\varepsilon $ be a sequence in 
 		$L^{B}\left( \Omega \right) $ such that $w_{\varepsilon }\to w$
 		strongly in $L^{B}\left( \Omega \right) ,$ then $\mathcal{T}_{\varepsilon
 		}\left( w_\varepsilon\right) \to w$ strongly in $L^{B}\left( \Omega \times
 		Y\right) .$
 		
 		\item[iii)] \, For every relatively weakly compact sequence $\left\{
 		w_{\varepsilon }\right\}_\varepsilon $ in $L^{B}\left( \Omega \right) $\ the
 		corresponding $\left\{ \mathcal{T}_{\varepsilon }\left( w_{\varepsilon
 		}\right) \right\}_\varepsilon $ is relatively weakly compact in $L^{B}\left( \Omega
 		\times Y\right) .$
 		Furthermore,
 		if $\mathcal{T}_{\varepsilon }\left( w_{\varepsilon }\right)
 		\rightharpoonup \widehat{w}$ weakly in $L^{B}\left( \Omega \times Y\right) ,$
 		then $w_{\varepsilon }\rightharpoonup \mathcal{M}_{Y}\left( \widehat{w}%
 		\right) $ weakly in $L^{B}\left( \Omega \right) $.
 		
 		\item[iv)] If $\mathcal{T}_{\varepsilon }\left( w_{\varepsilon
 		}\right) \rightharpoonup \widehat{w}$ weakly in $L^{B}\left( \Omega \times
 		Y\right) ,$ then 
 		\begin{align*}\left\Vert \widehat{w}\right\Vert _{L^{B}\left( \Omega
 			\times Y\right) }\leq \underset{\varepsilon \rightarrow 0}{\lim \inf }\left(
 		1+\left\vert Y\right\vert \right) \left\Vert w_{\varepsilon }\right\Vert
 		_{L^{B}\left( \Omega \right) }.
 		\end{align*}
 	\end{itemize}
 \end{theorem}

\section{Notation and Preliminaries}\label{sec:con2}

In what follows $X$ and
$V$ denote a locally compact space and a Banach space, respectively, and the spaces $L^p(X; V)$ and $L^p_{\rm loc}(X; V)$
($X$ provided with a positive Radon measure) are denoted by $L^p(X)$ and
$L^p_{\rm loc}(X)$, respectively, when $V = \mathbb R$. We refer to \cite{FLbook} for integration theory and more details.

In the sequel we denote by $Y$ the cube $]0,1[^d$, or, more generally, a similar reference cell or set with the paving property, e.g., any cube with sides parallel to the coordinate axes.\footnote{The Lebesgue measure of the standard unit cube mentioned above is 1, so, by considering only this case, some formulas may be simplified.}

%In order to enlighten the space variable under consideration we will adopt the notation $\mathbb R^d_x, \mathbb R^d_y$, or $\mathbb R^d_z$ to indicate where $x,y $ or $z$ belong to.

The family of open subsets in $\mathbb R^d$ will be denoted by $\mathcal A(\mathbb R^d)$, its subsets constituted by bounded sets %\color{blue} whose boundary has null Lebesgue measure (i.e. $\mathcal L^d(\partial \Omega))=0$) \color{black}  
is denoted by $\mathcal A_0(\mathbb R^d)$.

For any $N \in \mathbb N$, and any subset $E$ of $\mathbb R^N$, by $\overline E$, we denote its closure in the relative topology, by $\widehat E$ its interior, and by $\partial E$, its topological boundary.
For any set $E$, its characteristic function $\chi_E$ is defined as $\left\{\begin{array}{ll}
 1 &\hbox{ in  }E,\\
0 &\hbox{ otherwise.}\end{array}\right.$.

For every $x \in \mathbb R^d$ we denote by $[x]$ its integer part, namely the vector in $\mathbb Z^d$, which has as components the integer parts of the components of $x$.

By $\mathcal L^d$ we denote the Lebesgue measure in $\mathbb R^d$, hence for any $E \subset \mathbb R^d$, $\mathcal L^d(E)$ denotes its measure, but for the sake of shortening the notation, when $E=Y$, we will adopt the notation $|Y|$ instead of $\mathcal L^d(Y)$.

\subsection{Orlicz spaces}\label{O-spaces}

\bigskip Let $B:\left[ 0,+\infty \right[ \rightarrow \left[ 0,+\infty \right[
$ be an ${\rm N}-$function (see \cite{ada} for more results and proofs on $N$-functions and Orlicz spaces), i.e., $B$ is continuous, convex, with $%
B\left( t\right) >0$ for $t>0,\frac{B\left( t\right) }{t}\rightarrow 0$ as $t\rightarrow 0,$ and $\frac{B\left( t\right) }{t}\rightarrow \infty $ as $%
t\rightarrow \infty .$
Equivalently, $B$ is of the form $B\left( t\right)
=\int_{0}^{t}b\left( \tau \right) d\tau ,$ where $b:\left[ 0,+\infty \right[
\rightarrow \left[ 0,+\infty \right[ $ is non decreasing, right continuous,
with $b\left( 0\right) =0,b\left( t\right) >0$ if $t>0$ and $b\left(
t\right) \rightarrow +\infty $ if $t\rightarrow +\infty .$ 

We denote by $\widetilde{B},$ the complementary ${\rm N}-$function of $B$ defined by $$\widetilde{B}(t)=\sup_{s\geq 0}\left\{ st-B\left( s\right) ,t\geq 0\right\}.
$$ It follows
that 
\begin{equation}\nonumber 
	\frac{tb(t)}{B(t)} \geq 1
	\;\;(\hbox{or }> \hbox{if }b\hbox{ is strictly increasing}),
\end{equation}
\begin{equation}
	\nonumber \widetilde{B}( b(t) )\leq
	tb( t) \leq B( 2t) \hbox{ for all }t>0.
\end{equation}
An ${\rm N}-$function $B$ is of class $\triangle _{2}$ near $\infty$ (denoted  by $B\in \triangle
_{2}$) if there are $\alpha >0$ and $t_{0}\geq 0$ such that 
\begin{equation}\label{Delta2}B\left(
	2t\right) \leq \alpha B\left( t\right) 
\end{equation} for all $t\geq t_{0}$. 

\noindent In what
follows every $N$- function $B$  and its conjugate $\widetilde{B}$ 
satisfy the $\triangle_2$ condition, this latter property can be equivalently stated saying that $B$ satisfies the $\nabla_2$ condition.
In the sequel, unless otherwise specified, $c$ and $C$ refer to positive constants, which may vary from line to line. 

%It is also worth recalling that given two $N$-functions
%$B$ and $C$, $B$ dominates $C$ (denoted as $C \prec B$) if there is $k>0$ such that $C(t) \leq B(k t)$ for all $t \geq 0$. Hence, it follows that $C \prec B$ if and only if 
%${\widetilde B} \prec {\widetilde C}$.

Let $\Omega $ be a
bounded open set in $\mathbb R^d$, the Orlicz space
\begin{equation*}L^{B}\left(
	\Omega \right) =\left\{ u:\Omega \rightarrow 
	\mathbb R \hbox{ measurable},\lim_{\delta \to 0^+} \int_{\Omega
	}B\left( \delta \left\vert u\left( x\right) \right\vert \right) dx=0\right\} 
\end{equation*}
is a Banach space with respect to the Luxemburg norm: \begin{equation*}\left\Vert u\right\Vert
	_{B,\Omega }=\inf \left\{ k>0:\int_{\Omega }B\left( \frac{\left\vert u\left(
		x\right) \right\vert }{k}\right) dx\leq 1\right\}, %<+\infty .
	\end{equation*}
moreover, 
$L^{B}\left(\Omega\right)$ coincides with the set of measurable functions for which the Luxemburg norm is finite.

It follows
that: $\mathcal{D}\left( \Omega \right) $ is dense in $L^{B}\left( \Omega
\right)$, $L^{B}\left( \Omega \right) $ is separable and reflexive, the dual
of $L^{B}\left( \Omega \right) $ may be identified with $L^{\widetilde{B}}\left(
\Omega \right),$ with norm given by %$L^{\widetilde{B}}\left( \Omega \right) $
%is equivalent to 
$\left\Vert \cdot\right\Vert _{\widetilde{B},\Omega }.$
We will denote the norm of elements in $L^{B}\left( \Omega \right)$, both by $\|\cdot\|_{L^{B}\left( \Omega \right)}$ and with $\|\cdot\|_{B, \Omega}$, the latter symbol being useful when we want emphasize the domain $\Omega$.

Futhermore, it is also convenient to recall that:
\begin{itemize} 
	\item[(i)]\; $
	\left\vert \int_{\Omega }u\left( x\right) v\left( x\right) dx\right\vert
	\leq 2\left\Vert u\right\Vert _{B,\Omega }\left\Vert v\right\Vert _{%
		\widetilde{B},\Omega }$ for $u\in L^{B}\left( \Omega \right) $ and $v\in L^{%
		\widetilde{B}}\left( \Omega \right) $, 
	\item[(ii)]\; given $v\in L^{%
		\widetilde{B}}\left( \Omega \right) $ the linear functional $L_{v}$ on $%
	L^{B}\left( \Omega \right) $ defined by $L_{v}\left( u\right):=\int_{\Omega
	}u\left( x\right) v\left( x\right) dx,\left( u\in L^{B}\left( \Omega \right)
	\right) $ belongs to 
	%the dual $\left[ L^{B}\left( \Omega \right) \right]
	%^{\prime }=
	$L^{\widetilde{B}}\left( \Omega \right) $ with $\left\Vert
	v\right\Vert _{\widetilde{B},\Omega }\leq \left\Vert L_{v}\right\Vert _{\left[ L^{B}\left( \Omega \right) \right] ^{\prime }}\leq 2\left\Vert
	v\right\Vert_{\widetilde{B},\Omega }$, 
	\item[(iii)]\;  the property $\lim_{t \to +\infty} \frac{B\left( t\right) }{t}=+\infty $
	implies $L^{B}\left( \Omega \right) \subset L^{1}\left( \Omega \right)
	\subset L_{loc}^{1}\left( \Omega \right) \subset \mathcal{D}^{\prime }\left(
	\Omega \right),$ each embedding being continuous.
\end{itemize}

For the sake of notations, given any $m\in \mathbb N$, when $u:\Omega \to \mathbb R^m$, we mean that each component $(u^i)$, of $u$, $(1\leq i \leq m)$, lies in $L^B(\Omega)$ and  
we will denote the norm of $u$ with the symbol $\|u\|_{L^B(\Omega)^{d}}:=\sum_{i=1}^d \|u^i\|_{B,\Omega}$.

%Analogously one can define the Orlicz-Sobolev functional space as follows: 
%
%\noindent$%
%W^{1}L^{B}\left( \Omega \right) =\left\{ u\in L^{B}\left( \Omega \right) :%
%\frac{\partial u}{\partial x_{i}}\in L^{B}\left( \Omega \right) ,1\leq i\leq
%d\right\} ,$ where derivatives are taken in the distributional sense on $%
%\Omega .$ Endowed with the norm $\left\Vert u\right\Vert _{W^{1}L^{B}\left(
%	\Omega \right) }=\left\Vert u\right\Vert _{B,\Omega }+\sum_{i=1}^{d}$ $%
%\left\Vert \frac{\partial u}{\partial x_{i}}\right\Vert _{B,\Omega },u\in
%W^{1}L^{B}\left( \Omega \right) ,$\ \ $W^{1}L^{B}\left( \Omega \right) $ is
%a reflexive Banach space. We denote by $W_{0}^{1}L^{B}\left( \Omega \right)
%, $ the closure of $\ \mathcal{D}\left( \Omega \right) $\ in $%
%W^{1}L^{B}\left( \Omega \right) $ and the semi-norm $u\rightarrow \left\Vert
%u\right\Vert _{W_{0}^{1}L^{B}\left( \Omega \right) }=\left\Vert
%Du\right\Vert _{B,\Omega }=\sum_{i=1}^{d}$ $\left\Vert \frac{\partial u}{%
%	\partial x_{i}}\right\Vert _{B,\Omega }$ is a norm on $W_{0}^{1}L^{B}\left(
%\Omega \right) $ equivalent to $\left\Vert \cdot \right\Vert _{W^{1}L^{B}\left(
%	\Omega \right) }.$
%
%By $W_{\#}^{1}L^{B}\left( Y\right)$, we denote the space of functions $u \in W^1L^B(Y)$ such that $\int_Y u(y)d y=0$.  It is endowed with the gradient norm.

Given a function space $S$ defined in $Y$, 
%$Z$ or $Y\times Z$, 
the subscript $S_{per}$ means that its elements are $Y$-periodic.
%, $Z$ or $Y\times Z$, as it will be clear from the context. 
In particular $C_{per}(Y)$ 
%and $C_{per}(Y\times Z)$ 
denotes the space of $Y$-periodic functions in $C(\mathbb R^d)$.
%and $C(\mathbb R^d_y\times \mathbb R^d_z)$, resoectively, 
%i.e. that verify $w(y+k)$ for $y \in \mathbb R^d$ and $k \in \mathbb Z^d$.
% and  $w(y + k, z + h) = w(y, z)$ for $(y, z) \in \mathbb R^d \times \mathbb R^d$
%and $(k, h) \in \mathbb Z^d \times \mathbb Z^d$, respectively. 
$C^\infty_{per}(Y):= C_{per}(Y)\cap C^\infty(\mathbb R^d)$.
% and $C^\infty_{per}(Y\times Z)=C_{per}(Y\times Z)\cap C^\infty(\mathbb R^d_y\times \mathbb R^d_z)$, respcetively.  
For every $p\geq 1$, $L^p_{per}(Y)$, %$L^p_{per}(Y\times Z)$ are 
is the space of $Y$-periodic function in $L^p_{loc}(\mathbb R^d)$. 
%and $L^p_{loc}(\mathbb R^d_y \times \mathbb R^d_z)$, respectively. 
Analogously  $L^B_{per}(Y)$ is the space of $Y$-periodic functions in $L^B_{loc}(\mathbb R^d)$.
% and $L^B_{loc}(\mathbb R^d_y
%\times \mathbb R^d_z)$, respectively.

\section{The unfolding operator}\label{sec:con3}

Here and in the sequel, considering the above notation, we will consider $\Omega \in \mathcal A_0(\mathbb R^d)$.
Furthermore we assume that there exists a base $\mathcal B:=\{\bf b_1, \dots, \bf b_d\}\subset \mathbb R^d$, with the following property, i.e. for a.e. $z \in \mathbb R^d$ there exists unique $n_1,\dots, n_d \in \mathbb Z^d$ such that $z \in \xi + Y$, where $\xi = \sum_{h=1}^d n_h \bf b_h$. 

We set $\xi= [z]_Y$ and $\{z\}_Y:= z- [z]_Y\in Y$, hence, up to an affine transformation we can assume that $\mathcal B$ is the canonical basis in $\mathbb R^d$ and $\xi \in \mathbb Z^d$. In fact in most of the following results we will assume %to be a bounded open subset with Lipschitz boundary of $\mathbb R^d$
 $Y:=]0,1[^d$, $\mathcal B$ the canonical base and $[z]_Y$ the integer part of $z \in \mathbb R^d$, as in \cite{CDG2}. %\color{blue} and, unless otherwise specified, the positive constant $C$ may vary from line to line. \color{black} 
 %Following \cite{CDG2}, for $z\in 
%TCIMACRO{\U{211d} }%
%BeginExpansion
%\mathbb{R}^d ,\left[ z\right] _{Y}$ is the vector with components $\left[ z_{i}\right] 
%$ where $\left[ z_{i}\right] $\ is the integer part of $z_{i}.$ It follows
%that $z-\left[ z\right] _{Y}=\left\{ z\right\} _{Y}\in Y.$ 

\noindent Then, for each $%
x\in 
\mathbb{R}^d$, $x=\varepsilon\left( \left[ \frac{x}{\varepsilon }\right] _{Y}+\left\{ \frac{x}{%
	\varepsilon }\right\} _{Y}\right) .$ 

Define \begin{align}\label{2.1CDG2}
 \Xi_{\varepsilon }:=\left\{ \xi \in 
\mathbb{Z}
^{d},\;\varepsilon \left( \xi +Y\right) \subset \Omega \right\},\nonumber\\ 
\widehat{\Omega }_{\varepsilon }:= {\rm int} \left\{ \bigcup\limits_{\xi \in \Xi
	_{\varepsilon }}\varepsilon \left( \xi +\overline{Y}\right) \right\},\\
\Lambda _{\varepsilon }:=\Omega \backslash \widehat{\Omega }_{\varepsilon }.\nonumber
\end{align}
The set $\widehat \Omega_\varepsilon$  is the largest union of cells $\varepsilon(\xi + \overline Y)$ (with $\xi \in \mathbb Z^d$) included in $\Omega$ while $\Lambda_\varepsilon$ is
the subset of $\Omega$ containing the parts from cells $\varepsilon(\xi +\overline Y)$ intersecting the boundary $\partial \Omega$.

We are in position of introducing the unfolding operator, i.e. we recall \cite[Definition 2.1]{CDG2}, since it acts on Lebesgue measurable functions.

\begin{definition}\label{unfdef}
	For $\phi $ Lebesgue measurable on $\Omega$, the unfolding operator $%
	\mathcal{T}_{\varepsilon }$ is defined as 
	\begin{equation*}
		\mathcal{T}_{\varepsilon }\left( \phi \right) \left( x,y\right) =\left\{ 
		\begin{tabular}{ll}
			$\phi \left( \varepsilon \left[ \frac{x}{\varepsilon }\right]
			_{Y}+\varepsilon y\right) $ & a.e for $\left( x,y\right) \in \widehat{\Omega 
			}_{\varepsilon }\times Y$ \\ 
			$0$ & a.e for $\left( x,y\right) \in \Lambda _{\varepsilon }\times Y.$%
		\end{tabular}%
		\right.
	\end{equation*}%
	Clearly $\mathcal T_\varepsilon (\phi)$ is Lebesgue measurable in $\Omega \times Y$, and is $0$ whenever $x$ is outside $\widehat \Omega_\varepsilon$. Moreover, for every $v,w $ Lebesgue-measurable, $$\mathcal{T}_{\varepsilon }\left( vw\right) =\mathcal{T}%
	_{\varepsilon }\left( v\right) \mathcal{T}_{\varepsilon }\left( w\right)$$
\end{definition}

 Analogous identities hold for other pointwise operations on functions. In particular, the operator is linear with respect to pointwise operations.

\begin{proposition}
	Let $f\in L_{per}^{1}\left( Y\right) $ define the sequence $\left\{
	f_{\varepsilon }\right\}_\varepsilon $ by $f_{\varepsilon }\left( x\right) :=f\left( 
	\frac{x}{\varepsilon }\right) $ a.e for $x\in 
	%TCIMACRO{\U{211d} }%
	%BeginExpansion
	\mathbb{R}
	%EndExpansion
	^{n},$ then 
	\begin{equation*}
		\mathcal{T}_{\varepsilon }\left( f_{\varepsilon }|_{\Omega }\right) \left(
		x,y\right) =\left\{ 
		\begin{tabular}{ll}
			$f\left( y\right) $ & \;a.e for $\left( x,y\right) \in \widehat{\Omega }%
			_{\varepsilon }\times Y$, \\ 
			$0$ & \;a.e for $\left( x,y\right) \in \Lambda _{\varepsilon }\times Y.$%
		\end{tabular}%
		\right.
	\end{equation*}%
	If $f\in L_{per}^{B}\left( Y\right)$, %,and if $\Omega $\ is bounded, 
	$\mathcal{T}_{\varepsilon }\left( f_{\varepsilon }|_{\Omega }\right)
	\to f$ strongly in $L^{B}\left( \Omega \times Y\right) .$
\end{proposition}

\begin{proof} For any $k>0$ such that $\int_{Y}B\left( \frac{\left\vert f\right\vert }{k}\right) \left( y\right) dy\leq 1 ,$ observe that
	$B\left( \frac{\left\vert \mathcal{T}_{\varepsilon }\left( f_{\varepsilon
	}|_{\Omega }\right) -f\right\vert}{k} \right) \left( x,y\right) =\left\{ 
	\begin{tabular}{ll}
		$0$ & \;a.e for $\left( x,y\right) \in \widehat{\Omega }_{\varepsilon }\times
		Y$, \\ 
		$\frac{\left\vert f\left( y\right) \right\vert}{k} $ & \;a.e for $\left( x,y\right) \in
		\Lambda _{\varepsilon }\times Y.$%
	\end{tabular}%
	\right. $
	
\noindent Hence
\begin{align*}\iint_{\Omega \times Y}B\left(\frac{ \left\vert \mathcal{T}_{\varepsilon
}\left( f_{\varepsilon }|_{\Omega }\right) -f\right\vert}{k} \right) \left(
x,y\right) dxdy=\iint_{\Lambda _{\varepsilon }\times Y}B\left(\frac{ \left\vert
f\right\vert}{k} \right) \left( x,y\right) dxdy=\\
\mathcal L^d\left( \Lambda _{\varepsilon
}\right) \int_{Y}B\left( \frac{\left\vert f\right\vert}{k} \right) \left( y\right) dy.
\end{align*}

Since %\color{blue}$\mathcal L^d(\partial \Omega )=0,$ \color{black}
%is Lipschitz and bounded, 
$\mathcal L^d\left( \Lambda _{\varepsilon }\right)
\to 0$\ as $\varepsilon \to 0,$ recalling the definition of norm in $L^B$ spaces, $\mathcal{%
	T}_{\varepsilon }\left( f_{\varepsilon }|_{\Omega }\right) \to f$
strongly in $L^{B}\left( \Omega \times Y\right) .$ \end{proof}

 The following result holds. It can be immediately deduced by \cite[Proposition 2.5]{CDG2}, observing that for every $N$- function $B$ and $w \in L^B(\Omega)$ $B(\mathcal T_\varepsilon(w))= \mathcal T_\varepsilon B(w)$.
\color{black}

\begin{proposition}\label{Prop2}
	For every $N$-function $B$, the operator $\mathcal{T}_{\varepsilon }$\ is linear and
	continuous from $L^{B}\left( \Omega \right) $ to $L^{B}\left( \Omega \times
	Y\right) $. 
	It results that
	\begin{itemize}
		\item[i)] $\frac{1}{|Y|}\iint_{\Omega \times Y}
		 B(\mathcal T_\varepsilon (w))(x, y) dx dy =\int_\Omega B(w(x))dx- \int_{\Lambda_\varepsilon}B(w(x))dx=\int_{\widehat \Omega_\varepsilon} B(w(x))dx.$
	\\	\item[ii)]$ \frac{1}{|Y|}\iint_{\Omega \times Y}
	B(\mathcal T_\varepsilon (w))(x, y) dx dy \leq \int_\Omega B(w(x))dx$
	\\	\item[iii)] $\frac{1}{|Y|}|\iint_{\Omega \times Y}
	 B(\mathcal T_\varepsilon (w))(x, y) dx dy -\int_\Omega B(w(x))dx|\leq \int_{\Lambda_\varepsilon}B(w(x))dx$
	\\	\item[iv)]$\|\mathcal T_\varepsilon(w)\|_{L^B(\Omega \times Y)} = \|w \chi_{\widehat \Omega_\varepsilon}\|_{L^B(\Omega)}$,
	\end{itemize}	
with $\chi_{\widehat \Omega_\varepsilon}=\left\{\begin{array}{ll} 
	1 & \;\hbox{ in }\widehat \Omega_\varepsilon,\\
	0 & \;\hbox{otherwise},
	\end{array}\right.$
as in subsection \ref{O-spaces}.
	In particular, for every $\varepsilon, $
\begin{align}\label{est1}	\left\Vert \mathcal{T}%
	_{\varepsilon }\left( w\right) \right\Vert _{L^{B}\left( \Omega \times
		Y\right) }\leq \left( 1+\left\vert Y\right\vert \right) \left\Vert
	w\right\Vert _{L^{B}\left( \Omega \right) }.
	\end{align}
\end{proposition}

\begin{proof}
	Arguing as in \cite[(i) in Proposition 2.5]{CDG2},  it results that
	$$\frac{1}{|Y|}\iint_{\Omega \times Y} B(\mathcal T_\varepsilon (w))(x, y) dx dy =\int_\Omega B(w(x))dx- \int_{\Lambda_\varepsilon}B(w(x))dx=\int_{\widehat \Omega_\varepsilon} B(w(x))dx.$$
Indeed, by the definitions of unfolding operator $\mathcal T_\varepsilon$ and $\hat \Omega_\varepsilon$, it results \begin{align}
\frac{1}{|Y|} \iint_{\Omega \times Y}
	B(\mathcal T_\varepsilon(w))(x, y) dx dy =
	\frac{1}{|Y|}
	\iint_{\widehat \Omega_\varepsilon \times Y}
B(\mathcal T_\varepsilon(w))(x,y)dxdy= \nonumber\\
	\frac{1}{|Y|}\sum_{\xi \in \Xi_\varepsilon}\iint_{(\varepsilon \xi +\varepsilon Y)\times Y} B(\mathcal T_\varepsilon(w))(x,y)dxdy.\label{sumunf}
	\end{align}
	In particular $\mathcal T_\varepsilon(w)(x, y) = w(\varepsilon \xi +\varepsilon y)$ is constant in $x$ on each $(\varepsilon \xi +\varepsilon Y)\times Y$.
%\end{align*}
Consequently, each summand in the above equality becomes
\begin{align}
	\iint_{(\varepsilon \xi +\varepsilon Y)\times Y} B(\mathcal T_\varepsilon (w))(x,y)dx dy
	= \mathcal L^d(\varepsilon \xi + \varepsilon Y)
	\int_Y
	B(w(\varepsilon \xi + \varepsilon y) )dy \nonumber\\
	= \varepsilon^d|Y |
	\int_{\varepsilon \xi +\varepsilon Y} B(w)(\varepsilon \xi +\varepsilon y)dy =
	|Y |\int_Y B(w)(x)dx.	\label{estunf}
\end{align}
By summing over $\xi \in \Xi_\varepsilon$ in \eqref{sumunf} and exploiting \eqref{estunf} we obtain
\begin{align}\label{iv}
	\frac{1}{|Y|} \iint_{\Omega \times Y}
	B(\mathcal T_\varepsilon(w))(x, y) dx dy =\int_{\widehat \Omega_\varepsilon} B(w)(x)dx.
\end{align}

\noindent Thus, using the same argument, for every $k>0$ such that $\int_{\Omega
}B\left( \frac{\left\vert w\right\vert }{k }\right) \left( x\right)
dx \leq 1$,  exploiting the definition of Luxemburg norm, recalling that $B$ is convex and $B(0)=0$, 
	\begin{align*}
	\iint_{\Omega \times Y}B\left( \frac{\left\vert \mathcal{T}_{\varepsilon
		}\left( w\right) \right\vert }{\left( 1+\left\vert Y\right\vert \right)
		k}\right) \left( x,y\right) dxdy\leq \frac{1}{\left( 1+\left\vert
		Y\right\vert \right) }\iint_{\Omega \times Y}B\left( \frac{\left\vert 
		\mathcal{T}_{\varepsilon }\left( w\right) \right\vert }{k }\right)
	\left( x,y\right) dxdy\leq\\
	 \frac{1}{\left\vert Y\right\vert }\iint_{\Omega
		\times Y}B\left( \frac{\left\vert \mathcal{T}_{\varepsilon }\left( w\right)
		\right\vert }{k }\right) \left( x,y\right) dxdy\leq \int_{\Omega
	}B\left( \frac{\left\vert w\right\vert }{k}\right) \left( x\right)
	dx. 
	\end{align*} Therefore $\left\Vert \mathcal{T}_{\varepsilon }\left( w\right)
	\right\Vert _{L^{B}\left( \Omega \times Y\right) }\leq \left( 1+\left\vert
	Y\right\vert \right) \left\Vert w\right\Vert _{L^{B}\left( \Omega \right) }.$
	
	\noindent Observe that i), ii), iii) and  iv) follow by \eqref{iv}.
\end{proof}

From this result, in particular from  ii), it is possible to provide, as in the standard $L^p$ setting (cf. \cite[Proposition 2.6]{CDG2}), an unfolding criterion for integrals in the Orlicz setting.
 For the sake of a more complete parallel with the standard $L^p$ setting, we recall the  unfolding criterion for integrals,  u.c.i. for short, as introduced in \cite{CDG2}. 

\begin{proposition}\label{u.c.i.}
	If $\left\{ w _{\varepsilon }\right\}_\varepsilon $\ is a
	sequence in $L^{1}\left( \Omega \right) $ satisfying $\int_{\Lambda
		_{\varepsilon }}\left| w _{\varepsilon }\right| dx\to
	0$\ as $\varepsilon \to 0,$ then \begin{align*}\ \int_{\Omega }
	w_{\varepsilon } dx- \frac{1}{\left\vert Y\right\vert }%
	\iint_{\Omega \times Y}\mathcal{T}_{\varepsilon }\left( w _{\varepsilon
	}\right) dxdy\to 0\end{align*} as $\varepsilon \to 0$.
\end{proposition}
This result justifies the following notation for integrals of unfolding operators. Indeed, if $\left\{ w_{\varepsilon }\right\}_\varepsilon$ is a sequence
satisfying $u.c.i$, we write $$\int_{\Omega } w_{\varepsilon
} dx\overset{\mathcal{T}_{\varepsilon }}{\simeq }\frac{1}{%
	\left\vert Y\right\vert }\iint_{\Omega \times Y}\mathcal{T}_{\varepsilon
}\left( w_{\varepsilon }\right) dxdy.$$

\begin{proposition}[\bf u.c.i. in the Orlicz setting]\label{uciO}
	If $\left\{ w _{\varepsilon }\right\}_\varepsilon $\ is a
	sequence in $L^B\left( \Omega \right) $ satisfying $\int_{\Lambda
		_{\varepsilon }}B( w _{\varepsilon })dx\to
	0$\ as $\varepsilon \to 0,$ then $$\ \int_{\Omega }\
	B(w_{\varepsilon }) dx-\frac{1}{\left\vert Y\right\vert }%
	\iint_{\Omega \times Y}B\left(\mathcal{T}_{\varepsilon }\left( w _{\varepsilon
	}\right)\right) dxdy\to 0 $$\ as $\varepsilon \to 0$.
\end{proposition}
\begin{proof} The result is a consequence of iii) in Proposition \ref{Prop2}.
	\end{proof}
\begin{proposition}\label{Prop5}
	Let $\left\{ u_{\varepsilon }\right\}_\varepsilon $ be a bounded sequence in $
	L^{B}\left( \Omega \right) $ and $v\in L^{\widetilde{B}}\left( \Omega
	\right) $ then $$\int_{\Omega }u_{\varepsilon }vdx\overset{\mathcal{T}_{\varepsilon }}{\simeq } \frac{1}{\left\vert Y\right\vert }\iint_{\Omega
		\times Y}\mathcal{T}_{\varepsilon }\left( u_{\varepsilon }\right) \mathcal{T}%
	_{\varepsilon }\left( v\right) dxdy.$$
\end{proposition}

\begin{proof}
	To obtain the stated result, it is enough to prove that $\int_{\Lambda _{\varepsilon }}\left\vert u_{\varepsilon
	}v\right\vert dx\to 0$\ as $\varepsilon \to 0.$ 
By H\"older's inequality in Orlicz setting ((i) in subsection \ref{O-spaces}), it results that
\begin{align*}
	\int_{\Lambda _{\varepsilon }}\left\vert u_{\varepsilon }v\right\vert
	dx=\int_{\Omega }\left\vert \chi_{\Lambda _{\varepsilon }}u_{\varepsilon
	}v\right\vert dx\leq 2\left\Vert u_{\varepsilon }\right\Vert _{L^{B}\left(
		\Omega \right) }\left\Vert \chi_{\Lambda _{\varepsilon }}v\right\Vert _{L^{%
			\widetilde{B}}\left( \Omega \right) }\leq C\left\Vert \chi_{\Lambda
		_{\varepsilon }}v\right\Vert _{L^{\widetilde{B}}\left( \Omega \right) }.\end{align*}
	Observe that $\underset{\varepsilon \to 0}{\lim }\chi_{\Lambda
		_{\varepsilon }}\left( x\right) =0,$ for every $x \in \mathbb R^d$, hence for every $k >0,$
	\ \ $\underset{\varepsilon \to 0}{\lim }\widetilde{B}\left( 
	\frac{\left\vert \chi_{\Lambda _{\varepsilon }}v\right\vert }{k }\right)
	=0$ a.e. in $\Omega.$ Moreover $\exists k_{1}>0$ such that $\widetilde{B}\left( 
	\frac{\left\vert \chi_{\Lambda _{\varepsilon }}v\right\vert }{k_{1}}%
	\right) \leq \widetilde{B}\left( \frac{\left\vert v\right\vert }{k_1}
	\right) $ and $\int_{\Omega }\widetilde{B}\left( \frac{\left\vert
		v\right\vert }{k_{1}}\right) dx\leq 1.$
	
\noindent Recalling that $\tilde{B}$ satisfies the $\Delta_2$ condition,	by Lebesgue's dominated convergence theorem, and \cite[Lemma A.4]{CPST}, 
$
	\underset{\varepsilon \to 0}{\lim }\int_{\Omega }\widetilde{B}%
	\left( \frac{\left\vert \chi_{\Lambda _{\varepsilon }}v\right\vert }{k
		_{1}}\right) \left( x\right) dx=0,$ that is 
	%$\underset{\varepsilon
	%	\to 0}{\lim }\frac{\left\Vert \chi_{\Lambda _{\varepsilon
	%	}}v-0\right\Vert _{L^{\widetilde{B}}\left( \Omega \right) }}{\lambda _{1}}%
	%=0; $ equivalently 
	$\underset{\varepsilon \to 0}{\lim }%
	\left\Vert \chi_{\Lambda _{\varepsilon }}v\right\Vert _{L^{\widetilde{B}}\left(
		\Omega \right) }=0.$
\end{proof}

As in the classical Lebesgue setting we can define the mean value operator acting on $L^B$ spaces.

\begin{definition}\label{meandef}
	The mean value operator $\mathcal{M}_{Y}\colon L^{B}\left( \Omega
	\times Y\right) \to L^{B}\left( \Omega \right) $ is defined as
	follows $$\mathcal{M}_{Y}\left( w \right) \left( x\right) :=\frac{
		1}{\left\vert Y\right\vert }\int_Y w\left( x,y\right) dy$$ for
	a.e. $x\in \Omega$ and for every $w \in L^B(\Omega \times Y)$.
\end{definition}

\begin{remark}
	As a consequence $\left\Vert \mathcal{M}_{Y}\left( w\right) \right\Vert
	_{L^{B}\left( \Omega \right) }\leq \frac{1}{\sup\{1,|Y|^{-1}\}}%left\vert Y\right\vert ^{-1}
	\left\Vert
	w \right\Vert _{L^{B}\left( \Omega \times Y\right) },$ for every $w
	\in L^{B}\left( \Omega \times Y\right) .$ Indeed, using Jensen's inequality, we
	get for $k >0,$ \begin{align*}\int_{\Omega }B\left( \frac{
		\mathcal{M}_{Y}\left( w\right) }{k (\sup\{1,|Y|^{-1}\})}\right) dx &=\int_{\Omega
	}B\left( \int_{Y}\frac{w \left( x,y\right) }{k\left\vert Y\right\vert (\sup\{1,|Y|^{-1}\})}dy\right)
	dx\\
&\leq\frac{1}{|Y|(\sup\{1,|Y|^{-1}\})}\int_{\Omega }\int_{Y}B\left( \frac{w \left( x,y\right) 
}{k }\right) dxdy.\end{align*}from which, by definition of Luxemburg norm, the statement follows. 
%Hence $$\left\Vert \mathcal{M}_{Y}\left(w
%	\right) \right\Vert _{L^{B}\left( \Omega \right) }\leq \frac{1}{
%	\sup\{1,|Y|^{-1}\}}\left\Vert w \right\Vert _{L^{B}\left( \Omega \times
%		Y\right) }.
	%$$ 
	\color{black}
\end{remark}

Theorem \ref{convBunf}, that we restate here for the reader's convenience, extends to the Orlicz setting the correspective one in classical $L^p$ spaces, (cf. \cite[Proposition 2.9]{CDG2}). 
To this end, we recall
the convergence properties related to the unfolding operator
when $\varepsilon \to 0$, i.e.  for $w$ uniformly continuous on $\Omega$, with modulus of continuity $m_w$, it is
easy to see that
	\begin{align*}
		\sup_{x\in \widehat{\Omega}_\varepsilon,y \in Y}
|{\mathcal T}_\varepsilon(w)(x, y) -w(x)| \leq m_w(\varepsilon).
\end{align*}

{\bf Theorem}\ref{convBunf}
{\it	Let $B$ be an $N$-function satisfying $\nabla_2$ and $\Delta_2$ conditions. The following results hold:
\begin{itemize}
	\item[i)] For $w\in L^{B}( \Omega) ,\mathcal{T}_{\varepsilon }\left( w\right) \to w$ strongly in $L^{B}\left(
	\Omega \times Y\right) $.
	
	\item[ii)] Let $\left\{ w_{\varepsilon }\right\}_\varepsilon $ be a sequence in 
	$L^{B}\left( \Omega \right) $ such that $w_{\varepsilon }\to w$
	strongly in $L^{B}\left( \Omega \right) ,$ then $\mathcal{T}_{\varepsilon
	}\left( w_\varepsilon\right) \to w$ strongly in $L^{B}\left( \Omega \times
	Y\right) .$
	
	\item[iii)] \, For every relatively weakly compact sequence $\left\{
	w_{\varepsilon }\right\}_\varepsilon $ in $L^{B}\left( \Omega \right) $\ the
	corresponding $\left\{ \mathcal{T}_{\varepsilon }\left( w_{\varepsilon
	}\right) \right\}_\varepsilon $ is relatively weakly compact in $L^{B}\left( \Omega
	\times Y\right) .$
	Furthermore,
	if $\mathcal{T}_{\varepsilon }\left( w_{\varepsilon }\right)
	\rightharpoonup \widehat{w}$ weakly in $L^{B}\left( \Omega \times Y\right) ,$
	then $w_{\varepsilon }\rightharpoonup \mathcal{M}_{Y}\left( \widehat{w}%
	\right) $ weakly in $L^{B}\left( \Omega \right) $.
	
	\item[iv)] If $\mathcal{T}_{\varepsilon }\left( w_{\varepsilon
	}\right) \rightharpoonup \widehat{w}$ weakly in $L^{B}\left( \Omega \times
	Y\right) ,$ then 
	\begin{align}\label{est2}\left\Vert \widehat{w}\right\Vert _{L^{B}\left( \Omega
			\times Y\right) }\leq \underset{\varepsilon \rightarrow 0}{\lim \inf }\left(
		1+\left\vert Y\right\vert \right) \left\Vert w_{\varepsilon }\right\Vert
		_{L^{B}\left( \Omega \right) }.
	\end{align}
\end{itemize}}
\begin{proof}
	\begin{itemize}
		\item[(i)] It is worth, preliminarily, to observe that in the statement, with an abuse of notation, we identify $w \in L^B(\Omega)$ with the element $L^B(\Omega \times Y)\ni w(x,y):= w(x)$.  We start proving the result when $\varphi \in \mathcal D(\Omega)$.
%		$\varphi \in \mathcal{D}\left( \Omega \right) \otimes
%		L^{B}\left( Y\right) .$ set $\varphi =\varphi _{1}\varphi _{2},\varphi
%		_{1}\in \mathcal{D}\left( \Omega \right) ;\varphi _{2}\in L^{B}\left(
%		Y\right) .$ $\varphi _{1}\in \mathcal{D}\left( \Omega \right)
%		\to 
Thus it exists a compact set $K$ such that ${\rm supp} \varphi \subset K,$
		and there exists $\varepsilon _{0}>0,$ $0<\varepsilon <\varepsilon
		_{0}$ such that  $K\subset \subset \widehat{\Omega }_{\varepsilon }.$ Let $k>0,$ and let us prove that 
		$$\iint_{\Omega \times Y}B\left( \frac{%
			\left\vert \mathcal{T}_{\varepsilon }\left( \varphi \right) \left( x,y\right)
			-\varphi\left( x\right) \right\vert }{k}\right) dxdy\rightarrow 0,$$ as $\varepsilon \rightarrow 0.$
		
		For $0<\varepsilon <\varepsilon _{0},$ for every $x\in \Lambda _{\varepsilon }$ we have that $\varphi\left( x\right) =0,$ hence 
		%for $0<\varepsilon <\varepsilon _{0},$%
		\begin{align*}
		\iint_{\Omega \times Y}B\left( \frac{\left\vert \mathcal{T}%
					_{\varepsilon }\left( \varphi \right) \left( x,y\right) -\varphi \left(
					x\right) \right\vert }{k }\right)
				dxdy=\iint_{\widehat{\Omega }_{\varepsilon }\times Y}B\left( \frac{\left\vert \mathcal{T}_{\varepsilon }\left( \varphi\right)
					\left( x,y\right)
					-\varphi \left( x\right)\right\vert}{%
					k}\right) dxdy
\\
				=\sum\limits_{\xi \in \Xi _{\varepsilon}} \iint_{( \varepsilon
					\xi +\varepsilon Y) \times Y}B\left( \frac{\left\vert \varphi \left( \varepsilon \xi +\varepsilon y\right)  -\varphi\left(
					x\right) \right\vert }{%
					k }\right) dxdy\leq \sum\limits_{\xi \in \Xi_{\varepsilon }}\iint_{(\varepsilon \xi +\varepsilon Y)\times Y}B\left(
				 \frac{\left\vert m_{\varphi}\left(\varepsilon \right) \right\vert }{k}\right) dxdy.
				\end{align*}
		Since $m_{\varphi}\left( \varepsilon \right) \rightarrow 0,$ as $%
		\varepsilon \rightarrow 0$, there exists $0<\varepsilon _{1},\varepsilon
		<\varepsilon _{1}\to m_{\varphi}\left( \varepsilon \right)
		<1.$ Therefore, for $0<\varepsilon <\min \left( \varepsilon _{0},\varepsilon
		_{1}\right) ,$%
			 \begin{align*}
			 	\iint_{\Omega \times Y}B\left( \frac{\left\vert \mathcal{T}%
			 		_{\varepsilon }\left( \varphi \right) \left( x,y\right) -\varphi \left(
			 		x\right) \right\vert }{k }\right)
			 	dxdy\leq \\
			 	m_{\varphi}\left( \varepsilon {\rm diam}(Y)\right) \sum\limits_{\xi \in \Xi
					_{\varepsilon }}\iint_{( \varepsilon \xi +\varepsilon Y)
					\times Y}B\left( \frac{1}{k}\right) dxdy=\\ 
				 m_{\varphi}\left( \varepsilon {\rm diam }Y\right) \int_{\widehat{\Omega }%
					_{\varepsilon }}\int_{Y}B\left( \frac{1}{k }\right) dy\leq 
			C	m_{\varphi}\left( \varepsilon {\rm diam Y}\right) \mathcal L^d\left( \Omega \right)
				\int_{Y}{B}\left(\frac{1}
					{k}\right) dy \leq  \\ \mathcal L^d\left( \Omega \right) Cm_{\varphi}\left( \varepsilon {\rm diam} Y \right) \rightarrow 0,
				\end{align*} as $\varepsilon \rightarrow 0.$
%since clearly $k$ can be chosen such  that  $\int_{Y}B\left( \frac{1}{k }\right) dy<\infty.$

The general case follows from the density of $\mathcal D(\Omega)$ in $L^B(\Omega)$  a density argument, using the boundedness of the limit operator and the uniform boundedness of $\mathcal{T}_{\varepsilon}$ 

%((1.3) in Proposition 2.) 
%	and from the estimates

%		The result is obvious for any w ∈ D(Ω). If w ∈ Lp(Ω), let φ ∈ D(Ω).
%		Then, by using (iv) from Proposition 2.5,
%		
%		
%	\begin{align*}\|\mathcal T_{\varepsilon}(w) - w\|_{L^B(\Omega \times Y)} &= 
%\|\mathcal T_{\varepsilon}(w) - \mathcal T_\varepsilon (\phi) +\mathcal T_{\varepsilon}(\phi)- \phi + \phi- w\|_{L^B(\Omega \times Y)} \\
%&\leq \|T_{\varepsilon}(\phi)- \phi\|_{L^B(\Omega \times Y)}+ C(1+|Y|)\|\phi- w\|_{L^B(\Omega \times Y)}
%\end{align*}
%where $\phi \in \mathcal D(\Omega)$. Observe that the inequality is a consequence of the linearity of $\mathcal T_\varepsilon$ and 
\eqref{est1} in Proposition \ref{Prop2}. 
%Indeed, by the previous step
%we have
%\begin{align*}
%	\limsup_{\varepsilon \to 0}\|T_{\varepsilon}(w)- w\|
%\leq C(1+|Y|)\|\phi- w\|_{L^B(\Omega \times Y)},\end{align*}
%which concludes the proof of (i).

\item[(ii)] \; It is straightforward by the linearity of $\mathcal T_\varepsilon$ and \eqref{est1} in Proposition \ref{Prop2}.

\item[(iii)] \;  First we recall that since $B$ satisfies both $\nabla_2$ and $\Delta_2$ condition, the spaces $L^B$ are refelxive, cf. subsection \ref{O-spaces},\color{black} Then, by \eqref{est1} boundedness of $\{\mathcal T_\varepsilon(w_\varepsilon)\}_\varepsilon$ is preserved.
Moreover, taking into account that $w_\varepsilon \rightharpoonup \hat w$ in $L^{B}(\Omega)$, for every $v \in L^{\tilde B}(\Omega)$,   by Proposition \ref{Prop5}, $$\int_{\Omega }u_{\varepsilon }(x)v(x)dx\overset{\mathcal{T}_{\varepsilon }}{\simeq } \frac{1}{\left\vert Y\right\vert }\iint_{\Omega
	\times Y}\mathcal{T}_{\varepsilon }\left( u_{\varepsilon }\right) (x,y)\mathcal{T}%
_{\varepsilon }\left( v\right)(x,y) dxdy.$$
Hence, passing  to the limit as $\varepsilon \to 0$, one has
$$
\lim_{\varepsilon \to 0}
\int_\Omega
w_\varepsilon(x) v(x) dx =
\int_{\Omega}\left(
\frac{1}{|Y|} \int_Y {\hat w}(x, y) dy
\right)
v(x) dx,
$$
which, by Definition \ref{meandef} concludes the proof of (iii).

\item[(iv)] \;For what concerns the proof of \eqref{est2}, it results that it is a consequence of the lower semicontinuity of the norm with respect to the weak convergence and (ii) in Proposition \ref{Prop2}.
	\end{itemize}
Hence the proof is concluded.
	\end{proof}

\par\egroup

\end{document}